\documentclass[12pt,amsthm,amscd,amssymb,verbatim,epsf,amsmath,amsfonts]{amsart}

\theoremstyle{definition}

\usepackage[colorlinks=true,linkcolor=blue,citecolor=blue]{hyperref}
\usepackage{enumerate}

\def\line#1{\hbox to \hsize{#1\hfill}}
\usepackage{amsmath,amsfonts,amssymb,amscd,mathrsfs,graphicx,pst-node,tikz-cd,fancyvrb,lipsum,amsthm,graphicx,epsf,amsmath,verbatim}

\usepackage[colorlinks=true,linkcolor=blue,citecolor=blue]{hyperref}

\usepackage{times}
\usepackage{enumerate}
\begin{document}
\theoremstyle{plain}
\newtheorem{Thm}{Theorem}
\newtheorem{Cor}{Corollary}
\newtheorem{Con}{Conjecture}
\newtheorem{Main}{Main Theorem}
\newtheorem{Lem}{Lemma}
\newtheorem{Prop}{Proposition}
\def\R{\mathbb{R}}
\def\F{\mathbb{F}}
\def\Re{{\frak R\frak e}}
\def\Im{{\frak I\frak m}}
\def\S{\mathbb{S}}
\def\H{\mathbb{H}}
\def\L{\mathbb{L}}
\theoremstyle{definition}
\newtheorem{Def}{Definition}
\newtheorem{Note}{Note}

\newtheorem{example}{\indent\sc Example}

\theoremstyle{remark}
\newtheorem{notation}{Notation}
\renewcommand{\thenotation}{}

\errorcontextlines=0
\numberwithin{equation}{section}
\renewcommand{\rm}{\normalshape}%

\title{Minimal surfaces in the Riemannian product of surfaces}
\author{Nikos Georgiou}
\address{Nikos Georgiou\\
  Department of Mathematics\\
          South East Technological University\\
          Waterford\\
          Co. Waterford\\
          Ireland.}
\email{ngeorgiou@setu.ie}

\author{Brendan Guilfoyle}
\address{Brendan Guilfoyle\\
          School of STEM\\
          Munster Technological University, Kerry\\
          Tralee\\
          Co. Kerry\\
          Ireland.}
\email{brendan.guilfoyle@mtu.ie}

\date{\today}

\begin{abstract}
Minimal surfaces in the Riemannian product of surfaces of constant curvature have been considered recently, particularly as these products arise as spaces of oriented geodesics of 3-dimensional space-forms.  This papers considers more general Riemannian products of surfaces and explores geometric and topological restrictions that arise for minimal surfaces. We show that generically, a totally geodesic surface in a Riemannian product is locally either a slice or a product of geodesics.  If the Gauss curvatures of the factors are negative, it is proven that there are no minimal 2-spheres, while minimal 2-tori are Lagrangian with respect to both product symplectic structures. If the surfaces have non-zero bounded curvatures, we establish a sharp lower bound on the area of minimal 2-spheres and explore the properties of the Gauss and normal curvatures of general compact minimal surfaces. 
\end{abstract}

\maketitle
\section{Introduction}

The study of immersed surfaces in 4-manifolds is in its infancy, containing as it does all of the difficulties associated with curves knotted in 3-manifolds. In particular, the codimension of two introduces both computational and topological complications which must be geometrically controlled. Previously, various special classes of surfaces immersed in special Riemannian products have been investigated, for example products of 2-spheres \cite{CU} \cite{torurb}, of deSitter 2-spaces \cite{Ge2} and hyperbolic 2-spaces \cite{urbano2} - see also \cite{Ge1}. For such products of constant curvature surfaces, the Pl\"ucker embedding links the product geometry to the space of oriented geodesics of 3-dimensional spaces of constant curvature \cite{AGK} \cite{GG1} \cite{GK1} \cite{salvai0} \cite{salvai1}.

The purpose of this article is to investigate minimal surfaces in the Riemannian 4-manifold $\Sigma_1\times \Sigma_2$ endowed with the product metric $G=g_1\oplus g_2$, where $(\Sigma_1,g_1)$ and $(\Sigma_2,g_2)$ are more general Riemannian surfaces.   The associated complex structures $j_1,j_2$ on each surface yields two complex structures $J_1=j_1\oplus j_2$ and $J_2=-j_1\oplus j_2$ on the product such that the triples $(G,J_1,\Omega_1)$ and $(G,J_2,\Omega_2)$ are K\"ahler structures on $\Sigma_1\times \Sigma_2$. 

A minimal submanifold of a Riemannian manifold $(M,g)$ is a submanifold which is a critical point of the volume functional associated to $g$. Using the first variation formula, a minimal submanifold is characterized by the vanishing of the mean curvature vector, which in our case has values in the 2-dimensional normal bundle of the immersion. 

One class of minimal surfaces is the totally geodesic surfaces, on which not only the mean curvature vector vanishes, but the whole second fundamental form vanishes. Our first result identifies totally geodesic surfaces in the product metric:

\begin{Thm}\label{t:1}
Let $(\Sigma_1,g_1)$ and $(\Sigma_2,g_2)$ be two 2-manifolds with Gauss curvatures $K_1$ and $K_2$ respectively. If $F: S\rightarrow \Sigma_1\times\Sigma_2$ is a totally geodesic immersion of a surface and $K_1(F(p))\neq K_2(F(p))$ for any $p\in S$, then $F(S)$ is locally either a slice or the product of geodesics in $(\Sigma_1,g_1)$ and $(\Sigma_2,g_2)$.
\end{Thm}

Further examples of minimal surfaces are the complex surfaces (or complex curves in the language of algebraic geometry): surfaces whose tangent spaces are invariant under the complex structures $J_1$ or $J_2$. Minimal surfaces that are Lagrangian with respect to the symplectic structures $\Omega_1$ and $\Omega_2$ have also been an active area of research \cite{CU} \cite{Ge1} \cite{Ge2}. 

The following Theorem establishes the existence of a 1-parameter family of minimal surfaces in the product of 2-dimensional space forms that are neither complex nor Lagrangian, given a pair of solutions of the sinh-Gordon equation. 

\begin{Thm}\label{t:2}
If $X_1,X_2:{\mathbb C}\rightarrow {\mathbb R}$ are two smooth solutions of the sinh-Gordon equation:
\[
X_{z\bar z}+\frac{1}{2}\sinh (2X)=0,
\]
then there exists a 1-parameter family of minimal immersions $F_t:{\mathbb C}\rightarrow {\mathbb S}^2_p\times {\mathbb S}^2_p$, where ${\mathbb S}^2_p$ is a 2-dimensional real space form, so that the induced metrics $F_t^{\ast}G$ are all equal to $4\cosh(X_1+X_2)\cosh(X_1-X_2)|dz^2|$ with K\"ahler functions $C_1=\tanh(X_1-X_2)$ and $C_2=\tanh(X_1+X_2)$.
\end{Thm}
This generalizes Theorem 1 of \cite{torurb} from the product of 2-spheres to the product of any 2-dimensional space of constant curvature. Previously a similar result for minimal surfaces in ${\mathbb S}^2\times{\mathbb R}$ had been proven in \cite{HKS13}.

We next focus on compact minimal surfaces.

\begin{Thm}\label{t:3}
There are no minimal 2-spheres in the product of oriented Riemannian surfaces of negative curvatures.
\end{Thm}

For minimal 2-tori in the product of 2-manifolds of negative curvatures, the following holds.

\begin{Thm}\label{t:4}
If a 2-torus is minimally immersed in the product of surfaces of negative curvatures, then it is Lagrangian with respect to both symplectic structures $\Omega_1$ and $\Omega_2$.
\end{Thm}

Next, we provide a lower area bound for minimal 2-spheres.

\begin{Thm}\label{t:5}
Assume that $(\Sigma_1,g_1)$ and $(\Sigma_2,g_2)$ are not both flat, and that both have bounded Gauss curvatures $K_1$ and $K_2$. If $F:S\rightarrow \Sigma_1\times\Sigma_2$ is a minimal embedding of a 2-sphere $S$, then the area $|S|$ of $S$ satisfies
\begin{equation}\label{e:areainequality}
|S|\geq\frac{4\pi}{\max\{\displaystyle\sup_{x\in\Sigma_1}|K_1(x)|,\sup_{y\in\Sigma_2}|K_2(y)|\}}.
\end{equation}

\end{Thm}

Finally, the following theorem gives a rigidity result for the sum of the Gauss and normal curvatures of a compact minimal surface.

\begin{Thm}\label{t:6}
Let $F:S\rightarrow \Sigma_1\times\Sigma_2$ be a minimal immersion of an oriented compact surface with $M_j\neq 0$ (for the definition see equation (\ref{e:simplicity})). Then, the following statements hold.
\begin{enumerate}
\item If $K+(-1)^{j+1}K^\bot\geq 0$, then either $F$ is a complex curve with respect to the complex structure $J_j$ or $F$ is Lagrangian with respect to the symplectic structure $\Omega_j$.
\item If both $(\Sigma_1,g_1)$ and $(\Sigma_2,g_2)$ have negative curvatures, then $K+(-1)^{j+1}K^\bot<0$ cannot hold everywhere along $F$.
\end{enumerate}
\end{Thm}

In the next section the basic facts of product geometries are recalled, while the fundamental equations of minimal surfaces are given in Section \ref{s:3}. The proofs of Theorems \ref{t:1} and \ref{t:2} are also contained in Section \ref{s:3}. Section \ref{s:4} considers compact minimal surfaces and contains the proofs of Theorems \ref{t:3}, \ref{t:4}, \ref{t:5} and \ref{t:6}.

\vspace{0.1in}

\section{Preliminaries}

The following are well-known facts regarding Riemannian product manifolds - for the case of 2-dimensional products we follow the notation of \cite{torurb}.

Let $(\Sigma_1, g_1)$ and $(\Sigma_2, g_2)$ be two oriented Riemannian 2-manifolds and for $k\in\{1,2\}$, denote by $j_k$ the complex structure on $\Sigma_k$, defined as the rotation in $T\Sigma_k$ by an angle $+\pi/2$. Set $\omega_k(\cdot, \cdot)=g_k(j_k\cdot, \cdot)$, so that the quadruple $(\Sigma_k,g_k,j_k,\omega_k)$ defines a 2-dimensional K\"ahler manifold. 

Using the following identification, $X\in T(\Sigma_1\times\Sigma_2)\simeq X_1+X_2\in T\Sigma_1\oplus T\Sigma_2,$ where $X_k\in T\Sigma_k$, we obtain the natural splitting $T(\Sigma_1\times\Sigma_2)=T\Sigma_1\oplus T\Sigma_2$. For $(x,y)\in \Sigma_1\times\Sigma_2$ and $X=X_1+X_2,\, Y=Y_1+Y_2$ are tangent vectors in $T_{(x,y)}(\Sigma_1\times\Sigma_2)$, define the product metric $G$ by:
\[
G_{(x,y)}(X,Y)=g_1(X_1(x),Y_1(x))+g_2(X_2(y),Y_2(y)).
\]
The Levi-Civita connection $\nabla$ with respect to the metric $G$ is given by:
\[
\nabla_X Y=D^1_{X_1} Y_1+D^2_{X_2} Y_2,
\]
where $D^1$ and $D^2$ denote the Levi-Civita connections with respect to the metrics $g_1$ and $g_2$, respectively. 

Consider the endomorphisms $J_1,J_2\in \mbox{End}(T\Sigma_1\oplus T\Sigma_2)$ defined by $J_1=j_1\oplus j_2$ and $J_2=-j_1\oplus j_2$. Clearly, $J_1,J_2$ are almost complex structures on $\Sigma_1\times\Sigma_2$. In fact, the following holds:
\begin{Prop}\cite{Ge1}
The almost complex structures $J_1,J_2$ are both integrable and are therefore both complex structures on $\Sigma_1\times\Sigma_2$.
\end{Prop}

Let $\pi_k:\Sigma_1\times \Sigma_2\rightarrow \Sigma_k$ be the $k$-th projection, and define the following 2-forms 
\[
\Omega_k=(-1)^{k+1}\pi_1^{\ast}\omega_1+\pi_2^{\ast}\omega_2.
\]

In \cite{Ge1} it was proven that $(G,J_1,\Omega_1)$ and  $(G,J_2,\Omega_2)$ define two K\"ahler structures on $\Sigma_1\times\Sigma_2$ such that
\[
G=\Omega_1(J_1\cdot,\cdot)=\Omega_2(J_2\cdot,\cdot).
\] 
Let $P$ be defined by $P=J_1J_2=J_2J_1$. Then $P$ is an almost paracomplex structure on $\Sigma_1\times\Sigma_2$ that is parallel, i.e. $\nabla P=0$ and isometric, i.e. $G(P\cdot,P\cdot)=G(\cdot,\cdot)$. It is not hard to see that the splitting $X=X_1+X_2$, can be explicitly given by
\[
X_1=\frac{X+PX}{2},\qquad X_2=\frac{X-PX}{2}.
\]

Let $F:S\rightarrow M$ be an immersion of an oriented connected surface $S$ in $M$. Define the \emph{K\"ahler functions} $C_k$ on $S$ by:
\[
F^\ast\Omega_k=C_k\omega_S,\qquad k\in\{1,2\},
\]
where $\omega_S$ denotes the area form of $S$.

If $F_k=\pi_k\circ F:S\rightarrow \Sigma_k$ we write $F=(F_1,F_2)$. The \emph{Jacobians} of $F_k$ are defined by
\[
F^\ast_k\omega_k=\mbox{Jac}(F_k)\omega_S.
\]
\begin{Prop}
For $k\in\{1,2\}$ the following hold
\[
{\mbox{Jac}}(F_k)=\frac{C_1+(-1)^{k}C_2}{2}.
\]
\end{Prop}
\begin{proof}
From the K\"ahler function's definition we have
\begin{eqnarray}
C_k\omega_S&=&F^\ast\Omega_k\nonumber\\
&=&F^\ast((-1)^{k+1}\pi_1^\ast\omega_1+\pi_2^\ast\omega_2)\nonumber\\
&=&(-1)^{k+1}(F^\ast\circ\pi_1^\ast)\omega_1+(F^\ast\circ\pi_2^\ast)\omega_2\nonumber\\
&=&(-1)^{k+1}F_1^\ast\omega_1+F_2^\ast\omega_2\nonumber\\
&=&((-1)^{k+1}\mbox{Jac}(F_1)+\mbox{Jac}(F_2))\omega_S,\nonumber
\end{eqnarray}
which implies 
\[
C_k=(-1)^{k+1}\mbox{Jac}(F_1)+\mbox{Jac}(F_2),
\]
and the Proposition follows.
\end{proof}
\begin{Prop}\label{p:bothlagrangian}\cite{Ge1}
If $F:S\rightarrow \Sigma_1\times\Sigma_2$ is a Lagrangian immersion of a surface with respect to both symplectic structures $\Omega_1$ and $\Omega_2$, then $F(S)$ is locally the product $\gamma_1\times\gamma_2$, where $\gamma_k$ is a curve of $\Sigma_k$. If furthermore $F$ is minimal, then the curves $\gamma_k$ are geodesics. 
\end{Prop}

For any points $p\in\Sigma_1$ and $q\in\Sigma_2$ the associated slices are given by
\[
\Sigma_1\times\{q\}=\{(x,q)\in \Sigma_1\times\Sigma_2\,|\, x\in\Sigma_1\},
\]
and
\[
\{p\}\times\Sigma_2=\{(p,y)\in \Sigma_1\times\Sigma_2\,|\, y\in\Sigma_2\}.
\]
The slices are totally geodesic surfaces with $C_1^2=C_2^2=1$. 

\begin{Def}
An immersion $F:S\rightarrow M$ of a submanifold of half the real dimension of a complex manifold $(M,J)$ is said to be a \emph{complex submanifold} if $JTS=TS$. In the case where $M$ is of dimension 4, the complex surface is called a \emph{complex curve}.
\end{Def}

The following proposition classifies all surfaces that are complex curves with respect to both complex structures $J_1$ and $J_2$:
\begin{Prop}\label{p:totallygeodesics}
If $F:S\rightarrow \Sigma_1\times\Sigma_2$ is a complex immersion of a surface with respect to both complex structures $J_1$ and $J_2$, then $F(S)$ is locally the slice $\Sigma_1\times\{q\}$ or $\{p\}\times\Sigma_2$, for some points $p\in\Sigma_1$ and $q\in\Sigma_2$. 
\end{Prop}
\begin{proof}
Let $(x,y)$ be isothermal local coordinates of $(S,g)$. Since $F$ is a complex curve with respect to $J_1$ we have
\[
J_1F_x=a_1F_x+a_2F_y\qquad J_1F_y=b_1F_x+b_2F_y.
\]
Using the fact that $J_1^2=-Id$ we have 
\[
b_2=-a_1\qquad a_1^2+a_2b_1=-1.
\]
The compatibility condition $g(X,Y)=g(J_1X,J_1Y)$ for tangential vector fields $X,Y$ yields $J_1F_x=\pm F_y$. Assume without loss of generality that
\begin{equation}\label{e:comfx}
J_1F_x=F_y=((F_1)_y,(F_2)_y).
\end{equation}
On the other hand,
\[
J_1F_x=(j_1(F_1)_x,j_2(F_2)_x),
\]
and using equation (\ref{e:comfx}) we get
\begin{equation}\label{e:comfx1}
j_1(F_1)_x=(F_1)_y\qquad j_2(F_2)_x=(F_2)_y.
\end{equation}
The immersion $F$ is a complex curve with respect to $J_2$ and a similar argument yields $J_2F_x=\pm F_y$ and thus,
\[
j_1(F_1)_x=\mp (F_1)_y\qquad j_2(F_2)_x=\pm (F_2)_y.
\]
In case where $J_2F_x=F_y$ we have
\[
j_1(F_1)_x=-(F_1)_y\qquad j_2(F_2)_x=(F_2)_y,
\]
and together with equation (\ref{e:comfx1}) we obtain $(F_1)_x=(F_1)_y=0$ and therefore, $F_1(x,y)=p$ for some $p\in\Sigma_1$. 

Similar argument for the case $J_2F_x=-F_y$ implies that $F_2(x,y)=q$ for some $q\in\Sigma_2$ and the proposition follows.
\end{proof}
Let $F:S\rightarrow \Sigma_1\times\Sigma_2$  be an immersion of an oriented surface $S$ and denote by $g$ the induced metric $F^\ast G$.
Choose orthonormal frames $(s_1,s_2)$ and $(v_1,v_2)$ of $(\Sigma_1,g_1)$ and $(\Sigma_2,g_2)$, respectively, so that 
\[
j_1s_1=s_2\qquad j_2v_1=v_2.
\]
Suppose the derivative of $F_1$ is given by
\[
dF_1(e_1)=\lambda_1s_1+\lambda_2s_2,\qquad
dF_1(e_2)=\mu_1s_1+\mu_2s_2,
\]
and the derivative of $F_2$ by 
\[
dF_2(e_1)=\bar\lambda_1v_1+\bar\lambda_2v_2,\qquad
dF_2(e_2)=\bar\mu_1v_1+\bar\mu_2v_2.
\]

For $j\in\{1,2\}$, denote by $K_j:\Sigma_i\rightarrow {\mathbb R}$ the Gauss curvature of $(\Sigma_j,g_j)$. Then,
\[
R_1(s_1,s_2,s_2,s_1):=g_1(R_1(s_1,s_2)s_2,s_1)=K_1,
\]
and 
\[
R_2(v_1,v_2,v_2,v_1):=g_2(R_2(v_1,v_2)v_2,v_1)=K_2,
\]
where, $R_1$ and $R_2$ denote the corresponding Riemann curvature tensors. If $\bar R$ denotes the Riemann curvature tensor of the ambient metric $G$, and using the identification of $e_k$ with $dF(e_k)$ and $K_j$ with the real function $K_j\circ F_j$, we then have
\[
\bar R(e_1,e_2,e_2,e_1)=K_1(\lambda_1\mu_2-\lambda_2\mu_1)^2+K_2(\bar\lambda_1\bar\mu_2-\bar\lambda_2\bar\mu_1)^2.
\]

On the other hand,
\begin{eqnarray}
C_1&=&\omega_s(e_1,e_2)C_1\nonumber \\
&=&F^\ast\Omega_1(e_1,e_2)\nonumber \\
&=&\Omega_1((dF_1(e_1),dF_2(e_1)),(dF_1(e_2),dF_2(e_2)))\nonumber\\
&=&\omega_1(dF_1(e_1),dF_1(e_2))+\omega_2(dF_2(e_1),dF_2(e_2))\nonumber\\
&=&\omega_1(\lambda_1s_1+\lambda_2s_2,\mu_1s_1+\mu_2s_2)+\omega_2(\bar\lambda_1v_1+\bar\lambda_2v_2,\bar\mu_1v_1+\bar\mu_2v_2)\nonumber\\
&=&\lambda_1\mu_2-\lambda_2\mu_1+\bar\lambda_1\bar\mu_2-\bar\lambda_2\bar\mu_1.\nonumber
\end{eqnarray}
Similar calculations show 
\[
C_2=-\lambda_1\mu_2+\lambda_2\mu_1+\bar\lambda_1\bar\mu_2-\bar\lambda_2\bar\mu_1.
\]
Thus, 
\[
\bar R(e_1,e_2,e_2,e_1)=K_1\left(\frac{C_1-C_2}{2}\right)^2+K_2\left(\frac{C_1+C_2}{2}\right)^2.
\]
If $K$ denotes the Gauss curvature of $(S,g)$, the Gauss equation gives:
\begin{equation}\label{e:gausseqn}
K=K_1\left(\frac{C_1-C_2}{2}\right)^2+K_2\left(\frac{C_1+C_2}{2}\right)^2+2H^2-|h|^2/2,
\end{equation}
where $h,H$ are respectively the second fundamental form and the mean curvature of $F$.

Consider now isothermal coordinates $(x,y)$ of $(S,g)$, so that 
\[
g(F_x,F_x)=g(F_y,F_y)=e^{2u},\qquad g(F_x,F_y)=0,
\]
and $j$ is the associated complex structure so that $jF_x=F_y$. Note that $u=u(x,y)$ is a function defined locally on $S$. Consider an orthonormal frame $(N,\bar N)$ of the normal bundle such that $(F_x,F_y,N,\bar N)$ is an oriented orthonormal frame in $F^\ast T(\Sigma_1\times\Sigma_2)$. Let $z=x+iy$, and consider the complexified vector fields 
\[
F_z=\frac{1}{2}\left(F_x-iF_y\right)\qquad F_{\bar z}=\frac{1}{2}\left(F_x+iF_y\right).
\]
Then,
\[
G(F_z,F_z)=0\qquad G(F_z,F_{\bar z})=\frac{e^{2u}}{2}.
\]
If $\xi=\frac{N-i\bar N}{\sqrt{2}}$, we have that
\[
G(\xi,\xi)=0\qquad G(\xi,\bar\xi)=1.
\]
A brief calculation gives
\begin{equation}\label{e:complexcu1}
J_1F_z=iC_1F_z+\gamma_1\xi,
\end{equation}
\begin{equation}\label{e:complexcu2}
J_2F_z=iC_2F_z+\gamma_2\bar\xi,
\end{equation}
where $\gamma_1,\gamma_2:S\rightarrow {\mathbb C}$ are local complex functions on $S$ defined by
\[
\gamma_1=G(J_1F_z,\bar\xi)\qquad \gamma_2=G(J_2F_z,\xi).
\]
The fact that $J_k$ is $G$-compatible, yields
\begin{equation}\label{e:gamma1}
|\gamma_k|^2=\frac{e^{2u}(1-C_k^2)}{2}.
\end{equation}
From the above equations, one can see that $-1\leq C_k\leq 1$.

\begin{Prop}
The following statements are equivalent:
\begin{enumerate}
\item The immersion $F:S\rightarrow\Sigma_1\times\Sigma_2$ is a complex curve with respect to $J_k$.
\item The complex function $\gamma_k$ vanishes.
\item $C_k^2=1$. 
\end{enumerate}
\end{Prop}
\begin{proof}
The definition of a complex curve with the relations (\ref{e:complexcu1}) and (\ref{e:complexcu2}) shows that statements (1) and (2) are equivalent. On the other hand, the relation (\ref{e:gamma1}) shows that statements (2) and (3) are also equivalent and the proposition follows.
\end{proof}

\vspace{0.1in}


\section{Fundamental Equations of Minimal surfaces}\label{s:3}

Considering the Levi-Civita connection $\nabla$ of $(\Sigma_1\times\Sigma_2,G)$, define 
\[
F_{zz}:=\nabla_{F_z}F_z\qquad F_{z\bar z}:=\nabla_{F_{\bar z}}F_z= F_{\bar z z}\qquad  F_{\bar z\bar z}:=\nabla_{F_{\bar z}}F_{\bar z}.
\]
The Frenet equations of a minimal immersion $F$ are given by
\begin{eqnarray}
F_{zz}&=&2u_zF_z+f_1\xi+f_2\bar\xi\label{e:frenet1}\\
F_{z\bar z}&=&0\label{e:frenet2}\\
\xi_z&=&-2e^{-2u}f_2F_{\bar z}+A\xi\label{e:frenet3}\\
\bar\xi_z&=&-2e^{-2u}f_1F_{\bar z}-A\bar\xi\label{e:frenet4}
\end{eqnarray}
where $A,f_1,f_2$ are the complex functions on $S$, defined by 
\[
A=-G(\xi,\bar\xi_z)\qquad f_1=G(F_{zz},\bar\xi)\qquad f_2=G(F_{zz},\xi).
\]
For later use, we need the following expressions
\[
J_1\xi=-2e^{-2u}\bar\gamma_1F_z-iC_1\xi,
\]
and
\[
J_2\xi=-2e^{-2u}\bar\gamma_2F_{\bar z}+iC_2\xi.
\]

\begin{Def}
    Given a pair $(F,\xi)$ of a map and a complex normal frame, we call the triple $(A,\gamma_j,f_j)$ the \emph{fundamental data} of the pair.
\end{Def}

If $\{\xi^\ast,\bar\xi^\ast\}$ is another orthonormal oriented frame of the complexified normal bundle then $\xi^\ast=e^{i\theta}\xi$, for some real function $\theta$ on $S$ and let $(A^\ast,\gamma^\ast_j,f^\ast_j)$ be the fundamental data for the pair $(F,\xi^\ast)$. The fundamental data are related by
\[
\gamma_1^\ast=e^{-i\theta}\gamma_1\qquad \gamma_2^\ast=e^{i\theta}\gamma_2\qquad f_1^\ast=e^{-i\theta}f_1\qquad f_2^\ast=e^{i\theta}f_2\qquad A^\ast=i\theta_z+A.
\]
Thus $|\gamma_k|^2$ and $|f_k|^2$ are independent of the orthonormal frame choice of the normal bundle since,
$|\gamma^\ast_k|^2=|\gamma_k|^2$ and $|f^\ast_k|^2=|f_k|^2$.

For simplicity we introduce the real functions $M_k$ on $S$, where $k\in \{1,2\}$:
\begin{equation}\label{e:simplicity}
M_k=(-1)^{k+1}\frac{C_1-C_2}{2}K_1+\frac{C_1+C_2}{2}K_2.
\end{equation}

\begin{Prop}\label{p:derivatives}
Let $F:S\rightarrow \Sigma_1\times\Sigma_2$ be a minimal immersion of an orientable surface $S$. Then
\begin{eqnarray}
(C_k)_z&=&2ie^{-2u}\bar\gamma_kf_k,\label{e:cz}\\
(f_k)_{\bar z}&=&(-1)^{k+1}\bar Af_k+\frac{ie^{2u}\gamma_kM_k}{4},\label{e:f1z}\\
(\gamma_k)_{\bar z}&=&(-1)^{k+1} \bar A\gamma_k,\label{e:gamma1z}
\end{eqnarray}
where,
\begin{equation}\label{e:definofa}
A=2u_z-\frac{2iC_1f_1+(\gamma_1)_z}{\gamma_1}=-2u_z+\frac{2iC_2f_2+(\gamma_2)_z}{\gamma_2},
\end{equation}
provided that $\gamma_1\gamma_2\neq 0$.
\end{Prop}
\begin{proof}
The fact that $\nabla_{F_z}J_kF_z=J_kF_{zz}$ implies equations (\ref{e:cz}) and (\ref{e:definofa}). Differentiating equation (\ref{e:gamma1}) with respect to $z$ and using equation (\ref{e:definofa}), we obtain equation (\ref{e:gamma1z}).
Calculating the Riemann curvature tensor $\bar R$ of $G$, we have 
\begin{equation}\label{e:curv1}
\bar R(F_z,F_{\bar z})F_z=(-2u_{z\bar z}+2e^{-2u}|f_1|^2+2e^{-2u}|f_2|^2)F_z+(\bar Af_1-(f_1)_{\bar z})\xi -(\bar Af_2-(f_2)_{\bar z})\bar\xi.
\end{equation}
But on the other hand, we obtain
\begin{eqnarray}
\bar R(F_z,F_{\bar z})F_z&=& R_1((F_1)_z,(F_1)_{\bar z})(F_1)_z+ R_2((F_2)_z,(F_2)_{\bar z})(F_2)_z\nonumber\\
&=& K_1[g_1((F_1)_{\bar z},(F_1)_z)(F_1)_z-g_1((F_1)_z,(F_1)_z)(F_1)_{\bar z}]\nonumber\\
& & \qquad\qquad
+K_2[g_2((F_2)_{\bar z},(F_2)_z)(F_2)_z-g_2((F_2)_z,(F_2)_z)(F_2)_{\bar z}]\nonumber
\end{eqnarray}
\[
\qquad=\frac{K_1}{8}\left[G\left(F_{\bar z}+PF_{\bar z},F_z+PF_z\right)(F_z+PF_z)-
G\left(F_z+PF_z,F_z+PF_z\right)(F_{\bar z}+PF_{\bar z})
\right]
\]
\[
\qquad\quad+\frac{K_2}{8}\left[G\left(F_{\bar z}-PF_{\bar z},F_z-PF_z\right)(F_z-PF_z)-
G\left(F_z-PF_z,F_z-PF_z\right)(F_{\bar z}-PF_{\bar z})\right],
\]
implying that
\[
\bar R(F_z,F_{\bar z})F_z=\frac{K_1e^{2u}}{8}\left[((1-C_1C_2)^2-4e^{-4u}|\gamma_1|^2|\gamma_2|^2)F_z+i\gamma_1(C_2-C_1)\xi+i\gamma_2(C_1-C_2)\bar\xi\right]
\]
\[
\qquad\qquad\quad+\frac{K_2e^{2u}}{8}\left[((1+C_1C_2)^2-4e^{-4u}|\gamma_1|^2|\gamma_2|^2)F_z-i\gamma_1(C_1+C_2)\xi-i\gamma_2(C_1+C_2)\bar\xi\right].
\]
Thus
\begin{eqnarray}
\bar R(F_z,F_{\bar z})F_z&=&\frac{e^{2u}}{8}[K_1(C_1-C_2)^2+K_2(C_1+C_2)^2]F_z\nonumber \\
&&\qquad +\frac{i\gamma_1 e^{2u}}{8}[K_1(C_2-C_1)-K_2(C_1+C_2)]\xi\nonumber \\
&&\qquad\qquad +\frac{i\gamma_2 e^{2u}}{8}[K_1(C_1-C_2)-K_2(C_1+C_2)]\bar\xi.\label{e:curv2}
\end{eqnarray}
Finally, equation (\ref{e:f1z}) is obtained by equating equations (\ref{e:curv1}) and (\ref{e:curv2}).
\end{proof}

\vspace{0.1in}

Consider Euclidean 3-space ${\mathbb R}^3$ endowed with the flat metric $\left<.,.\right>_q$of signature $p\times (3-q)$ where $q\in\{0,1,2,3\}$. Let ${\mathbb S}_q^2$ be the 2-dimensional real space form defined by 
\[
{\mathbb S}_q^2=\{x\in {\mathbb R}^3\,|\, \left<x,x\right>_q=1\}.
\]
Regarding minimal surfaces in the product of 2-dimensional real space forms, we have the following:

\begin{Prop}
Let $S$ be a simply connected 2-manifold, $u,C_j$ be real smooth functions on $S$ with $C^2\leq 1$ where $C^2=1$ in at most isolated points of $S$ and $A,\gamma_j,f_j$ be smooth complex  functions on $S$ satisfying equation (\ref{e:gamma1}) and  Proposition \ref{p:derivatives}. Then there exists a (unique up to isometry) non-complex minimal immersion $F:S\rightarrow  {\mathbb S}_p^2\times  {\mathbb S}_p^2$ and an orthonormal reference $\{\xi,\bar\xi\}$ of the complexified normal bundle whose fundamental data are $(A,\gamma_j,f_j)$.  
\end{Prop}
\begin{proof}
If $F:S\rightarrow {\mathbb S}_q^2\times {\mathbb S}_q^2$ is a minimal immersion. then it can be extended to an immersion $\Phi:S\rightarrow {\mathbb R}^3\times {\mathbb R}^3:x\mapsto (\Phi_1(x),\Phi_2(x))$, so that $\Phi=\iota\circ F$, where $\iota$ is the inclusion map of ${\mathbb S}_q^2$ into  $({\mathbb R}^3, \left<.,.\right>_q)$. Let $\hat\Phi:S\rightarrow {\mathbb R}^3\times {\mathbb R}^3$ be a mapping defined by $\hat\Phi=(\Phi_1,-\Phi_2)$. Given the product metric $\left<\left<.,.\right>\right>_q:=\pi_1^\ast\left<.,.\right>_q+\pi_2^\ast\left<.,.\right>_q$, where $\pi_k:{\mathbb R}^3\times {\mathbb R}^3\rightarrow {\mathbb R}^3:(x_1,x_2)\mapsto x_k$ be the $k$-projection, then $\left<\left<\Phi,\Phi\right>\right>_q=\left<\left<\hat\Phi,\hat\Phi\right>\right>_q=2$ with $\left<\left<\Phi,\hat\Phi\right>\right>_q=0$, and therefore, $\{\Phi,\hat\Phi\}$ form an orthogonal reference along $\Phi$ of the normal bundle of ${\mathbb S}_q^2\times {\mathbb S}_q^2$ in ${\mathbb R}^6$. If $z=x+iy$ is a local isothermal parameter of $S$ then $\left<\left<\Phi_z,\Phi_z\right>\right>_q=0$ and $\left<\left<\Phi_z,\Phi_{\bar z}\right>\right>_q^2=\frac{e^{2u}}{2}$.

Note that $\hat\Phi_z=J_1J_2\Phi_z$. In particular we have
\[
\hat\Phi_z=-C_1C_2\Phi_z-2e^{-2u}\gamma_1\gamma_2\Phi_{\bar z}+iC_2\gamma_1\xi+iC_1\gamma_2\bar\xi.
\]
The Frenet equations of $\Phi$ are now given by:
\begin{eqnarray}
\Phi_{zz}&=&2u_z\Phi_z+f_1\xi+f_2\bar\xi+\frac{\gamma_1\gamma_2}{2}\hat\Phi,\label{e:frenet12}\\
\Phi_{z\bar z}&=&-\frac{e^{2u}}{4}\Phi+\frac{C_1C_2e^{2u}}{4}\hat\Phi,\label{e:frenet22}\\
\xi_z&=&-2e^{-2u}f_2F_{\bar z}+A\xi-\frac{iC_1\gamma_2}{2}\hat\Phi,\label{e:frenet32}\\
\bar\xi_z&=&-2e^{-2u}f_1F_{\bar z}-A\bar\xi-\frac{iC_2\gamma_1}{2}\hat\Phi.\label{e:frenet42}
\end{eqnarray}
The rest of the proof is identical to the proof of Proposition 2 in \cite{torurb}.
\end{proof}

We now return to the general case, that of a minimal immersion $F:S\rightarrow \Sigma_1\times\Sigma_2$ of an oriented surface $S$. The $F_z$-components of $\bar R(F_z,F_{\bar z})F_z$ from equations (\ref{e:curv1}) and (\ref{e:curv2}) yield
\begin{equation}\label{e:causscurv}
K=-4e^{-4u}(|f_1|^2+|f_2|^2)+\frac{K_1}{4}(C_1-C_2)^2+\frac{K_2}{4}(C_1+C_2)^2,
\end{equation}
where $K$ is the Gauss curvature of the $(S,g)$. Let $K^\bot$ be the normal curvature of $S$ in $\Sigma_1\times\Sigma_2$ and let $(e_1,e_2)$, $(v_1,v_2)$ be orthonormal frames of the tangent and normal bundles respectively, so that, $(e_1,e_2,v_1,v_2)$ agrees with the orientation of $\Sigma_1\times\Sigma_2$. Then
\[
K^\bot=\bar R^\bot (e_1,e_2,v_2,v_1).
\]
\begin{Prop}
The normal curvature  $K^\bot$ of a minimal immersion in $\Sigma_1\times\Sigma_2$ is given by
\begin{equation}\label{e:normalcurv}
K^\bot=-4e^{-4u}(|f_1|^2-|f_2|^2)+\frac{(K_1+K_2)(C^2_1-C^2_2)}{4}.
\end{equation}
\end{Prop}
\begin{proof}
Let $(e_1,e_2)$, $(v_1,v_2)$ be orthonormal frames of the tangent and normal bundles respectively so that $(e_1,e_2,v_1,v_2)$ agrees with the orientation of $\Sigma_1\times\Sigma_2$. The Ricci equation gives,
\[
\bar R^\bot(e_1,e_2,v_2,v_1)=\bar R(e_1,e_2,v_2,v_1)+G([A_{v_2},A_{v_1}]e_1,e_2),
\]
and setting $e^\pm_i=(e_i\pm Pe_i)/2$ and $v^\pm_i=(v_i\pm Pv_i)/2$, we get
\vspace{3pt}

\begin{eqnarray}
\bar R(e_1,e_2,v_2,v_1)&=&R_1(e^+_1,e^+_2,v^+_2,v^+_1)+R_2(e^-_1,e^-_2,v^-_2,v^-_1)\nonumber\\
&=&K_1\left(g_1(e^+_1,v^+_1)g_1(e^+_2,v^+_2)-g_1(e^+_2,v^+_1)g_1(e^+_1,v^+_2)\right)\nonumber\\
&&\qquad+K_2\left(g_2(e^-_1,v^-_1)g_2(e^-_2,v^-_2)-g_2(e^-_2,v^-_1)g_2(e^-_1,v^-_2)\right)\nonumber \\
&=&K_1G\left(\frac{e_1+Pe_1}{2},\frac{v_1+Pv_1}{2}\right)G\left(\frac{e_2+Pe_2}{2},\frac{v_2+Pv_2}{2}\right)\nonumber \\
&&-K_1G\left(\frac{e_2+Pe_2}{2},\frac{v_1+Pv_1}{2}\right)G\left(\frac{e_1+Pe_2}{2},\frac{v_2+Pv_2}{2}\right)\nonumber\\
&&+K_2G\left(\frac{e_1-Pe_1}{2},\frac{v_1-Pv_1}{2}\right)G\left(\frac{e_2-Pe_2}{2},\frac{v_2-Pv_2}{2}\right)\nonumber \\
&&-K_2G\left(\frac{e_2-Pe_2}{2},\frac{v_1-Pv_1}{2}\right)G\left(\frac{e_1-Pe_2}{2},\frac{v_2-Pv_2}{2}\right),\nonumber
\end{eqnarray}
which gives,
\[
\bar R(e_1,e_2,v_2,v_1)=\frac{K_1+K_2}{4}\left(G(e_1,Pv_1)G(e_2,Pv_2)-G(e_2,Pv_1)G(e_1,Pv_2)\right).
\]
Define $e_1=e^{-u}(F_z+F_{\bar z}),\, e_2=ie^{-u}(F_z-F_{\bar z}),\, v_1={\mbox Re}(\xi),\, v_2={\mbox Im}(\xi)$. Thus
\[
G(e_1,Pv_1)G(e_2,Pv_2)=\frac{e^{-2u}}{4}\left[C_2^2(\gamma_1-\bar\gamma_1)^2-C_1^2(\gamma_2-\bar\gamma_2)^2\right],
\]
and 
\[
G(e_1,Pv_2)G(e_2,Pv_1)=\frac{e^{-2u}}{4}\left[C_2^2(\gamma_1+\bar\gamma_1)^2-C_1^2(\gamma_2+\bar\gamma_2)^2\right].
\]
This yields
\begin{eqnarray}
\bar R(e_1,e_2,v_2,v_1)&=&\frac{(K_1+K_2)(C_1^2-C_2^2)}{8}.\label{e:riem00}
\end{eqnarray}
Computing $G([A_{v_2},A_{v_1}]e_1,e_2)$, we get
\begin{eqnarray}
G([A_{v_2},A_{v_1}]e_1,e_2)&=&-2e^{-4u}(|f_1|^2-|f_2|^2).\label{e:riem000}
\end{eqnarray}
Using equations (\ref{e:riem00}) and (\ref{e:riem000}), we have
\begin{eqnarray}
K^\bot &=&\frac{(K_1+K_2)(C_1^2-C_2^2)}{4}-4e^{-4u}(|f_1|^2-|f_2|^2),\nonumber
\end{eqnarray}
and the proposition follows.
\end{proof}

Using equations (\ref{e:causscurv}) and (\ref{e:normalcurv}), we get
\begin{equation}\label{e:f1length}
|f_j|^2=\frac{e^{4u}}{8}\left[-K+(-1)^jK^\bot+C_jM_j\right].
\end{equation}
\begin{Prop}\label{p:laplacianandnorm}
For $j\in\{1,2\}$, the gradient of $C_j$ is
\[
|\nabla C_j|^2=(1-C_j^2)\left(-K+(-1)^jK^\bot+C_jM_j\right),
\]
and the Laplacian is
\[
\Delta C_j=2C_j(K+(-1)^{j+1}K^\bot)-(1+C_j^2)M_j.
\]
\end{Prop}
\begin{proof}
Using equation (\ref{e:cz}) we obtain
\begin{eqnarray}
|\nabla C_1|^2&=&8e^{-4u}(1-C_1^2)|f_1|^2\nonumber, 
\end{eqnarray}
and from equation (\ref{e:f1length}) we get
\[
|\nabla C_1|^2=(1-C_1^2)\left(\frac{C_1}{2}\left[K_1(C_1-C_2)+K_2(C_1+C_2)\right]-K-K^\bot\right).
\]
Similarly, we get
\[
|\nabla C_2|^2=(1-C_2^2)\left(\frac{C_2}{2}\left[-K_1(C_1-C_2)+K_2(C_1+C_2)\right]-K+K^\bot\right).
\]
Considering the Laplacian $\Delta$ we have
\begin{eqnarray}
\Delta C_1&=& 8ie^{-2u}(-2u_{\bar z}e^{-2u}\bar\gamma_1f_1+e^{-2u}(\bar\gamma_1)_{\bar z}f_1+e^{-2u}\bar\gamma_1(f_1)_{\bar z}).\nonumber
\end{eqnarray}
The relation (\ref{e:definofa}) becomes,
\[
(\bar\gamma_1)_{\bar z}=2u_{\bar z}\bar\gamma_1+2iC_1\bar f_1-\bar A\bar\gamma_1,
\]
and using equation (\ref{e:f1z}) we get
\[
\Delta C_1=-16e^{-4u}C_1|f_1|^2-\frac{1-C_1^2}{2}\left[(C_1-C_2)K_1+(C_1+C_2)K_2\right],
\]
and combining this with equation (\ref{e:f1length}), we finally get
\[
\Delta C_1=2C_j(K+K^\bot)-\frac{1+C_1^2}{2}\left[K_1(C_1-C_2)+K_2(C_1+C_2)\right].
\]
Similar arguments shows 
\[
\Delta C_2=2C_j(K-K^\bot)-\frac{1+C_2^2}{2}\left[-K_1(C_1-C_2)+K_2(C_1+C_2)\right].
\]
\end{proof}

The following recovers the main theorem of \cite{Ge1}:

\begin{Prop}\label{t:totallygeodesictheorem}
Consider a pair of 2-dimensional oriented manifolds $(\Sigma_1,g_1)$ and $(\Sigma_2,g_2)$ such that their Gauss curvatures $K_1$ and $K_2$ satisfy $K_1\neq K_2$. Then every Lagrangian minimal surface in $\Sigma_1\times\Sigma_2$ is locally the product $\gamma_1\times\gamma_2$, where $\gamma_k$ is a geodesic in $(\Sigma_k,g_k)$.
\end{Prop}
\begin{proof}
If $C_1=0$ we have $\Delta C_1=0$ and therefore 
\[
C_2(K_1-K_2)=0.
\] 
This implies that $C_2=C_1=0$ which means that $F$ is Lagrangian minimal with respect to both symplectic structures. The theorem follows.
\end{proof}
\vspace{0.1in}
\noindent{\bf Theorem 1.}
{\it Let $(\Sigma_1,g_1)$ and $(\Sigma_2,g_2)$ be two Riemannian 2-manifolds with Gauss curvatures $K_1$ and $K_2$ respectively. If $F: S\rightarrow \Sigma_1\times\Sigma_2$ is a totally geodesic immersion of a surface, so that $K_1(F_1(p))\neq K_2(F_2(p))$ for any $p\in S$, then $F(S)$ is locally either a slice or the product of geodesics in $(\Sigma_k,g_k)$.
}
\begin{proof}
Since $F$ is totally geodesic, the second fundamental form $h$ must vanishes. Thus,
\[
f_1=G(F_{zz},\bar\xi)=0 \qquad f_2=G(F_{zz},\xi)=0
\]
This means that $|\nabla C_j|=0$ and therefore, $C_1$ and $C_2$ are both constant. This means that $\Delta C_j=0$ and thus
\[
2C_j(K+(-1)^{j+1}K^\bot)-\frac{1+C_j^2}{2}\left[(-1)^{j+1}K_1(C_1-C_2)+K_2(C_1+C_2)\right]=0.
\]
We then get
\begin{equation}\label{e:equation1.1}
(C_j^2-1)[K_1(C_1-C_2)+(-1)^{j+1}K_2(C_1+C_2)]=0.
\end{equation}
We now consider the following cases:

$\bullet$ $C_j^2< 1$, for any $j\in\{1,2\}$. In this case, using equation (\ref{e:equation1.1}) together with the assumption that $K_1^2+K_2^2= 0$ in at most isolated points, we conclude that $C_1=C_2=0$. This means that $F$ is locally the product of geodesics.

\vspace{0.1in}

$\bullet$ $C_1^2=1$ and $C_2^2<1$. It is not hard to see that
\begin{equation}\label{e:import1}
K_1(C_1-C_2)-K_2(C_1+C_2)=0.
\end{equation}
The fact that $C_2$ is constant implies $|\nabla C_2|^2=0$. Thus,
\[
\frac{C_2}{2}[-K_1(C_1-C_2)+K_2(C_1+C_2)]-K+K^\bot=0.
\] 
and therefore, $K=K^\bot$.

The expressions of $K$ and $K^\bot$ in equations (\ref{e:gausseqn}) and (\ref{e:normalcurv}) now become
\begin{equation}\label{e:gausstotal}
K=\frac{K_1}{4}(C_1-C_2)^2+\frac{K_2}{4}(C_1+C_2)^2,
\end{equation}
and 
\begin{equation}\label{e:normaltotal}
K^\bot=\frac{(K_1+K_2)(C_1^2-C_2^2)}{4}.
\end{equation}
The relations (\ref{e:import1}) imply,
\[
C_2(K_1+K_2)=C_1(K_1-K_2),
\]
and under the assumptions $C_1^2=1$ and $K_1\neq K_2$, we have that $C_2\neq 0$ and $K_1+K_2\neq 0$. Namely,
\begin{equation}\label{e:import2}
C_2=\frac{C_1(K_1-K_2)}{(K_1+K_2)}.
\end{equation}
Using equation (\ref{e:import2}) the Gauss curvature $K$ and normal curvature $K^\bot$ given in equations (\ref{e:gausstotal}) and (\ref{e:normaltotal}) become
\[
K=\frac{K^2_1+K^2_2-K_1K_2}{K_1+K_2}\qquad\qquad K^\bot=\frac{K_1K_2}{K_1+K_2},
\]
and using $K=K^\bot$ we then have
\[
K^2_1+K^2_2-K_1K_2=K_1K_2,
\]
yielding $K_1=K_2$ giving a contradiction. Thus, $C_1^2=C_2^2=1$ and therefore $F(S)$ is a complex curve with respect to both complex structures $J_1$ and $J_2$ and thus it is either the slice $\{p\}\times\Sigma_2$ or $\Sigma_1\times\{q\}$.
\end{proof}


We also give the following Proposition:
\begin{Prop}\label{p:laplacian}
Away from complex points, i.e. $C_j^2<1$, we have
\[
\Delta\log(1\pm C_j)=\mp M_j+K+(-1)^{j+1}K^\bot,
\]
where $M_j$ is defined in equation (\ref{e:simplicity}).
\end{Prop}
\begin{proof}
If $f$ is a smooth positive function on $S$, then
\[
\Delta\log f=\frac{f\Delta f-|\nabla f|^2}{f^2}.
\]
Then using Proposition \ref{p:laplacianandnorm} we have
\begin{eqnarray}
\Delta\log(1\pm C_j)
&=&\pm\frac{2C_j(K+(-1)^{j+1}K^\bot)-(1+C_j^2)M_j}{1\pm C_j}\nonumber\\
&&\qquad -\frac{(1-C_j^2)(C_jM_j-K+(-1)^jK^\bot)}{(1\pm C_j)^2}\nonumber\\
&=&\pm\frac{2C_j(K+(-1)^{j+1}K^\bot)-(1+C_j^2)M_j}{1\pm C_j}\nonumber\\
&&\qquad +\frac{(-1\pm C_j)(C_jM_j-K+(-1)^jK^\bot)}{1\pm C_j}\nonumber \\
&=&K+(-1)^{j+1}K^\bot+\frac{\mp1-C_j}{1\pm C_j}M_j\nonumber \\
&=&K+(-1)^{j+1}K^\bot\mp M_j,\nonumber
\end{eqnarray}
and the proposition follows.
\end{proof}

We now obtain the following:

\begin{Prop}
Let $F:S\rightarrow\Sigma_1\times\Sigma_2$ be a minimal immersion of an orientable surface without complex points. Then there exist smooth functions $v,w:S\rightarrow {\mathbb R}$ such that
\[
v_{z\bar z}+\frac{K_2|G(J_1F_z,J_2F_z)|}{4}\sinh 2v=0\quad w_{z\bar z}+\frac{K_1|G(J_1F_z,J_2F_z)|}{4}\sinh 2w=0,
\]
where $K_1,K_2$ denote the Gauss curvatures of $\Sigma_1$ and $\Sigma_2$, respectively.
\end{Prop}
\begin{proof}
Note that 
\[
|G(J_1F_z,J_2F_z)|=\frac{e^{2u}}{2}\sqrt{(1-C_1^2)(1-C_2^2)}.
\]
Define the functions $v,w:S\rightarrow {\mathbb R}$ by
\[
4v=\log\left(\frac{(1+C_1)(1+C_2)}{(1-C_1)(1-C_2)}\right)\qquad 
4w=\log\left(\frac{(1-C_1)(1+C_2)}{(1+C_1)(1-C_2)}\right).
\]
We then have $C_1=\tanh(v-w)$ and $C_2=\tanh(v+w)$. Using Proposition \ref{p:laplacian}, we get
\[
\Delta v=-\frac{K_2}{2}(C_1+C_2)\qquad\qquad \Delta w=-\frac{K_1}{2}(-C_1+C_2).
\]
We also have
\[
e^{2u}=2|G(J_1F_z,J_2F_z)|\cosh(v+w)\cosh(v-w).
\]
Thus
\[
4e^{-2u}v_{z\bar z}=\Delta v =-\frac{K_2}{2}(\tanh(v-w)+\tanh(v+w)),
\]
and therefore
\[
v_{z\bar z}=-\frac{K_2 e^{2u}\sinh2v}{8\cosh(v+w)\cosh(v-w)}=-\frac{K_2|G(J_1F_z,J_2F_z)| \sinh2v}{4},
\]
implying that
\[
v_{z\bar z}+\frac{K_2|G(J_1F_z,J_2F_z)|}{4}\sinh 2v=0.
\]
Similarly, one can show that
\[
w_{z\bar z}+\frac{K_1|G(J_1F_z,J_2F_z)|}{4}\sinh 2w=0.
\]
\end{proof}
We now have:

\vspace{0.1in}
\noindent{\bf Theorem 2.}
{\it If $X_1,X_2:{\mathbb C}\rightarrow {\mathbb R}$ are two smooth solutions of the sinh-Gordon equation:
\[
X_{z\bar z}+\frac{1}{2}\sinh (2X)=0,
\]
there exists a 1-parameter family of minimal immersions $F_t:{\mathbb C}\rightarrow {\mathbb S}^2_p\times {\mathbb S}^2_p$, where ${\mathbb S}^2_p$ is a 2-dimensional real space form, so that the induced metrics $F_t^{\ast}G$ are all equal to $4\cosh(X_1+X_2)\cosh(X_1-X_2)|dz^2|$ with K\"ahler functions $C_1=\tanh(X_1-X_2)$ and $C_2=\tanh(X_1+X_2)$.
}
\begin{proof}
Define local real functions $u,C_1,C_2:{\mathbb C}\rightarrow {\mathbb R}$ by
\[
C_1=\tanh(X_1-X_2)\qquad C_2=\tanh(X_1+X_2),
\]
and
\[
e^{2u}=4\cosh(X_1+X_2)\cosh(X_1-X_2).
\]
We also define the complex functions $\gamma_j,f_j,A:{\mathbb C}\rightarrow {\mathbb C}$ by
\[
\gamma_1=\left(2e^{it}\frac{\cosh(X_1+X_2)}{\cosh(X_1-X_2)}\right)^{1/2}\qquad \gamma_2=2e^{it}\gamma_1^{-1},
\]
and
\[
A=\frac{1}{2}\frac{\partial}{\partial z}\left[\log\left(\frac{\cosh(X_1+X_2)}{\cosh(X_1-X_2)}\right)\right].
\]
Then, $|\gamma_j|^2=\frac{e^{2u}(1-C_j^2)}{2}$ and $(\gamma_j)_{\bar z}=(-1)^{j+1}\bar A\gamma_j$. 
The complex functions $f_j:{\mathbb C}\rightarrow {\mathbb C}$ defined by
\[
f_j=-i\gamma_j(v+(-1)^jw)_z,
\]
satisfy, $(C_j)_z=2ie^{-2u}f_j\bar\gamma_j$. On the other hand,
\begin{eqnarray}
(f_j)_{\bar z}&=&-i(\gamma_j)_{\bar z}(v+(-1)^jw)_z-i\gamma_j(v_{z\bar z}+(-1)^jw_{z\bar z})\nonumber\\
&=&(-1)^{j+1}\bar A f_j+\frac{i\gamma_j}{2}\left(\sinh2v+(-1)^{j}\sinh2w\right).\nonumber
\end{eqnarray}
Using the relations
\[
\sinh 2v=\cosh(v-w)\cosh(v+w)(C_1+C_2),
\]
and 
\[
\sinh 2w=\cosh(v-w)\cosh(v+w)(-C_1+C_2),
\]
we have,
\[
(f_j)_{\bar z}=(-1)^{j+1}\bar A f_j+\frac{ie^{2u}\gamma_j}{8}\left((-1)^{j+1}(C_1-C_2)+(C_1+C_2)\right).
\]
Thus the functions $(u,A,C_j,\gamma_j,f_j)$ satisfy the structure equations for the minimal immersion $F_t$ for each $t\in {\mathbb R}$. It is not hard to compute that the induced metric of those immersions is $4\cosh(v+w)\cosh(v-w)|dz^2|$ and this completes the proof of the theorem.
\end{proof}

\vspace{0.1in}


\section{Compact Minimal Surfaces}\label{s:4}

We now study compact oriented 2-manifolds without boundary that are minimally immersed on $\Sigma_1\times\Sigma_2$. Such surfaces will be simply called compact minimal surfaces.

A classical theorem says that ${\mathbb R}^3$ doesn't contain compact minimal surfaces. Furthermore no simply-connected compact manifold can be minimally immersed in a complete Cartan-Hadamard manifold \cite{RS}. 

On the other hand, it is well known that every oriented, compact two dimensional manifold of genus$\geq 2$ admits a metric $g_j$ of negative curvature. In this case, the corresponding slice is a totally geodesic compact surface in $\Sigma_1\times\Sigma_2$ of genus$\geq 2$. Also, the Lyusternik–Fet theorem \cite{LF} says that $(\Sigma_j,g_j)$ admits a closed geodesic $\gamma_j$. 
If  $\Sigma_1$ and $\Sigma_2$ are both compact of genus $\geq 2$, they then admit metrics $g_1$ and $g_2$ of negative curvature and $\Sigma_1\times\Sigma_2$ contains a minimal 2-torus $\gamma_1\times\gamma_2$. 

The following theorem shows the non-existence of minimal 2-spheres when $g_j$ are of negative curvatures.

\vspace{0.1in}
\noindent{\bf Theorem 3.}
{\it There are no minimal 2-spheres in the product of oriented Riemannian surfaces of negative curvatures.
}
\begin{proof}
Suppose first $C_j^2<1$, for $j\in\{1,2\}$. Thus Proposition \ref{p:laplacianandnorm} gives
\[
|\nabla C_j|^2=(1-C_j^2)\left(C_j M_j-K+(-1)^jK^\bot\right)\geq 0,
\]
and therefore, 
\[
K+(-1)^{j+1}K^\bot\leq C_j M_j.
\]
Then
\[
\sum_{j=1}^2K+(-1)^{j+1}K^\bot\leq \sum_{j=1}^2 C_j M_j,
\]
and using the fact that $K_j\leq 0$ we finally get
\[
2K\leq \frac{(C_1-C_2)^2}{2}K_1+\frac{(C_1+C_2)^2}{2}K_2\leq 0.
\]
Assume without loss of generality that $C_1^2=1$ and $C_2^2<1$. Then
\[
|\nabla C_2|^2=(1-C_2^2)\left(C_2M_2-K+K^\bot\right)\geq 0,
\]
implying 
\begin{equation}\label{e:ineq}
K-K^\bot\leq \frac{C_2^2-C_1C_2}{2}K_1+\frac{C_2^2+C_1C_2}{2}K_2.
\end{equation}
On the other hand, Proposition \ref{p:laplacianandnorm} gives
\[
0=\Delta C_1=2C_1(K+K^\bot)-2M_1,
\]
which yields,
\[
K^\bot=-K+ \frac{C_1^2-C_1C_2}{2}K_1+\frac{C_1^2+C_1C_2}{2}K_2.
\]
The above relation combined with inequality  (\ref{e:ineq}) implies
\[
K\leq \frac{(C_1-C_2)^2}{4}K_1+\frac{(C_1+C_2)^2}{4}K_2\leq 0.
\]
Finally, consider the case $C_1^2=C_2^2=1$. Using Proposition \ref{p:laplacianandnorm} we have
\[
\Delta C_1=\Delta C_2=0,
\]
and a straightforward calculation gives
\[
K=\frac{(C_1-C_2)^2}{4}K_1+\frac{(C_1+C_2)^2}{4}K_2\leq 0.
\]
In any case we have $K\leq 0$. Therefore, the Euler number $\chi$ of $S$ is 
\[
\chi=\frac{1}{2\pi}\int_{S}KdA\leq 0,
\]
and the Theorem follows.
\end{proof}

Regarding minimal 2-tori in $\Sigma_1\times\Sigma_2$, where $\Sigma_1$ and $\Sigma_2$ are of negative curvatures, we have the following:

\vspace{0.1in}
\noindent{\bf Theorem 4.}
{\it If a 2-torus is minimally immersed in the product of surfaces of negative curvatures, then it is Lagrangian with respect to both symplectic structures $\Omega_1$ and $\Omega_2$.
}
\begin{proof}
Following similar arguments to those of the proof of Theorem \ref{t:3}, one can prove
\[
K\leq \frac{\lambda}{2}\left(C^2_1+C^2_2\right)\leq 0.
\]
On the other hand,
\[
0=\chi\leq \frac{\lambda}{4\pi}\int_{S}\left(C^2_1+C^2_2\right)\,dA\leq 0,
\]
and using the fact that $\lambda<0$ we have that $C_1=C_2=0$. The Theorem then follows from Proposition \ref{p:bothlagrangian}.
\end{proof}

The following theorem provides a lower bound for the area of minimal 2-spheres in $\Sigma_1\times\Sigma_2$,

\vspace{0.1in}
\noindent{\bf Theorem 5.}
{\it For  $j\in\{1,2\}$, assume that $(\Sigma_j,g_j)$ are not both flat surfaces and both have bounded Gauss curvatures. If $F:S\rightarrow \Sigma_1\times\Sigma_2$ is a minimal embedding of a 2-sphere $S$ in $\Sigma_1\times\Sigma_2$,  then the area $|S|$ of $S$ satisfies
\[
|S|\geq\frac{4\pi}{\max\{\displaystyle\sup_{x\in\Sigma_1}|K_1(x)|,\sup_{y\in\Sigma_2}|K_2(y)|\}},
\]
where $K_j$ is the Gauss curvature of $g_j$.
}
\begin{proof}
Using $|f_j|^2=\frac{e^{4u}}{8}\left(-K+(-1)^jK^\bot+C_jM_j\right)\geq 0$, we get
\[
K+(-1)^{j+1}K^\bot\leq C_jM_j,
\]
which implies 
\[
2K\leq C_1M_1+C_2M_2,
\]
and therefore, for any $p\in S$, we have 
\begin{eqnarray}
K(p)&\leq & \frac{(C_1(p)-C_2(p))^2}{4}K_1(p)+\frac{(C_1(p)+C_2(p))^2}{4}K_2(p)\nonumber\\
&\leq & \frac{(C_1(p)-C_2(p))^2}{4}|K_1(p)|+\frac{(C_1(p)+C_2(p))^2}{4}|K_2(p)|\nonumber\\
&\leq & \frac{(C_1(p)-C_2(p))^2}{4}\sup_{x\in\Sigma_1}|K_1(x)|+\frac{(C_1(p)+C_2(p))^2}{4}\sup_{y\in\Sigma_2}|K_2(y)|\nonumber\\
&\leq & \frac{1-C_1(p)C_2(p)}{2}\sup_{x\in\Sigma_1}|K_1(x)|+\frac{1+C_1(p)C_2(p)}{2}\sup_{y\in\Sigma_2}|K_2(y)|\nonumber \\
&\leq & \max\{\sup_{x\in\Sigma_1}|K_1(x)|,\sup_{y\in\Sigma_2}|K_2(y)|\}.\nonumber
\end{eqnarray}
Thus
\begin{eqnarray}
2&=&\frac{1}{2\pi}\int_S\,K\,dA\nonumber\\
&\leq &\frac{1}{2\pi}\int_S \max\{\sup_{x\in\Sigma_1}|K_1(x)|,\sup_{y\in\Sigma_2}|K_2(y)|\}\,dA\nonumber\\
&= &\frac{\max\{\displaystyle\sup_{x\in\Sigma_1}|K_1(x)|,\displaystyle\sup_{y\in\Sigma_2}|K_2(y)|\}\, |S|}{2\pi},\nonumber
\end{eqnarray}
and the theorem follows.
\end{proof}

The inequality (\ref{e:areainequality}) is sharp. In fact, it was proven in \cite{torurb} that for compact minimal surfaces in the product of the round 2-spheres ${\mathbb S}^2\times {\mathbb S}^2$,  equality holds when the surface is a slice which is a 2-sphere and the area is $4\pi$.

\vspace{0.1in}

In the case of ${\mathbb S}^2\times {\mathbb S}^2$, a rigidity result for the functions $K\pm K^\bot$ was given in \cite{torurb}. A similar result is now obtained for the product $\Sigma_1\times\Sigma_2$.

\vspace{0.1in}
\noindent{\bf Theorem 6.}
{\it Let $F$ be an immersion of an oriented compact surface in $\Sigma_1\times\Sigma_2$, where $M_j\neq 0$. Then, the following statements hold.
\begin{enumerate}
\item If $K+(-1)^{j+1}K^\bot\geq 0$, then either $F$ is a complex curve with respect to the complex structure $J_j$ or $F$ is Lagrangian with respect to the symplectic structure $\Omega_j$.
\item If $K+(-1)^{j+1}K^\bot<0$ and, furthermore, assume that both $\Sigma_1$ and $\Sigma_2$ have negative curvatures, then $F$ is a complex curve with respect to the complex structure $J_j$.
\end{enumerate}
}
\begin{proof}
A brief computation shows 
\[
C_jM_j=K+(-1)^{j+1}K^\bot+8e^{-4u}|f_j|^2.
\]
On the other hand,
\[
|\nabla C_j|^2=8e^{-4u}|f_j|^2(1-C_j^2),
\]
and 
\[
C_j\Delta C_j=-(1-C_j^2)\left(K+(-1)^{j+1}K^\bot\right)-8e^{-4u}(1+C_j^2)|f_j|^2.
\]
This means 
\begin{equation}\label{e:lap}
C_j\Delta C_j+|\nabla C_j|^2=-(1-C_j^2)\left(K+(-1)^{j+1}K^\bot\right)-16e^{-4u}C_j^2|f_j|^2.
\end{equation}

\vspace{0.1in}

\noindent \textbf{(1)} Assuming $K+(-1)^{j+1}K^\bot\geq 0$, we have that 
\[
C_j\Delta C_j+|\nabla C_j|^2\leq 0.
\]
The divergence theorem gives
\[
\int_{S}\left(C_j\Delta C_j+|\nabla C_j|^2\right) dA=0,
\]
and therefore
\[
C_j\Delta C_j+|\nabla C_j|^2= 0,
\]
which means that
\[
(1-C_j^2)\left(K+(-1)^{j+1}K^\bot\right)+16e^{-4u}C_j^2|f_j|^2=0.
\]
But since $(1-C_j^2)\left(K+(-1)^{j+1}K^\bot\right)\geq 0$ and $16e^{-4u}C_j^2|f_j|^2\geq 0$, then we get
\[
(1-C_j^2)\left(K+(-1)^{j+1}K^\bot\right)=16e^{-4u}C_j^2|f_j|^2=0.
\]
If $C_j^2=1$, then $|f_j|^2=0$ and therefore
\[
K+(-1)^{j+1}K^\bot=C_jM_j.
\]
If $C_j^2<1$, then $K+(-1)^{j+1}K^\bot=0$ and thus
\[
C_jM_j=K+(-1)^{j+1}K^\bot+8e^{-4u}|f_j|^2=8e^{-4u}|f_j|^2.
\]
We also have 
\[
16e^{-4u}C_j^2|f_j|^2=2C_j^3M_j=0,
\]
and since $M_j\neq 0$ we have $C_j=0$.

\vspace{0.1in}

\noindent \textbf{(2)} Consider the case $K(x)+(-1)^{j+1}K^\bot(x)<0$ for every $x\in S$. Assuming $C_j$ is constant with $C_j^2<1$, relation (\ref{e:cz}) gives $f_j=0$ and (\ref{e:f1z}) implies $M_j=0$.
Then
\[
|\nabla C_j|^2=(1-C_j^2)(C_jM_j-K+(-1)^jK^\bot)=0
\]
yielding a contradiction since we get
\[
K+(-1)^{j+1}K^\bot=0.
\]
Thus, $C_j^2=1$ and therefore $F$ is a complex curve with respect to $J_j$.

Assume $C_j$ is non-constant. In this case $C_j$ can't be constant in any open neighborhood either since otherwise, following similar argument we show that $F$ is a complex curve in that neighborhood and therefore it is a complex curve everywhere contradicting with the non-constant assumption.

Considering a critical point $p\in S$ of $C_j$ we obtain
\[
|\nabla C_j(p)|^2=(1-C_j^2(p))(C_j(p)M_j(p)-K+(-1)^jK^\bot)=0
\]
and assuming $C_j^2<1$, we have
\begin{equation}\label{e:inq}
C_j(p)M_j(p)=K(p)+(-1)^{j+1}K^\bot(p)<0.
\end{equation}
Now
\begin{eqnarray}
\Delta C_j(p)&=&2C_j(p)(K(p)+(-1)^{j+1}K^\bot(p))-(1+C_j^2(p))M_j(p)\nonumber \\
&=&2C_j(p)^2M(p)-(1+C_j^2(p))M_j(p)\nonumber \\
&=&-(1-C_j^2(p))M_j(p)\nonumber, 
\end{eqnarray}
and combined with inequality (\ref{e:inq}) we get that
\begin{equation}\label{e:inq1}
C_j(p)\Delta C_j(p)>0.
\end{equation}
Let $p_{min}$ and $p_{max}$ be the points of $S$ where $C_j$ is respectively minimum and maximum. Then $\Delta C_j(p_{min})>0$ and $\Delta C_j(p_{max})<0$. This means from inequality (\ref{e:inq1}) that $C_j(p_{max})<0<C_j(p_{min})$, which gives a contradiction.  Thus $C_j^2(p)=1$ for every critical point $p\in S$. 

We conclude that the only critical points of $C_j$ are the maxima and minima and Morse theory tells us that $S$ is a topological type of 2-sphere contradicting Theorem \ref{t:3}.
\end{proof}

\begin{Note}
Consider closed surfaces $(\Sigma_1,g_1)$ and $(\Sigma_2,g_2)$ with genus$\geq 2$ and metrics $g_1,g_2$ having constant negative curvature. For fixed $(p_1,p_2)\in\Sigma_1\times\Sigma_2$, Proposition \ref{p:totallygeodesics} tells us that the slices $\{p_1\}\times\Sigma_2$ and $\Sigma_1\times \{p_2\}$ are totally geodesic. Furthermore, the normal curvature $K^\bot$ vanishes and therefore for both slices the expression $K+(-1)^{j+1}K^\bot$ is negative.
\end{Note}

\vspace{0.1in}
\noindent{\bf Statements and Declarations:}

 The authors have no financial or non-financial interests that are directly or indirectly related to this work. No data was collected during the course of this work.

\end{document}